\documentclass[11pt]{amsart}
\usepackage[margin=1in]{geometry}
\usepackage{xcolor}

\usepackage[colorlinks=true]{hyperref}
\usepackage{amsmath}
\usepackage{amsfonts}
\usepackage{paralist}
\usepackage{multirow}
\usepackage{array}
\numberwithin{equation}{section}
\newtheorem{thm}{Theorem}[section]
\newtheorem{prop}[thm]{Proposition}
\newtheorem{lemma}[thm]{Lemma}

\newtheorem{defn}[thm]{Definition}

\newcommand{\ot}{\Omega_T }

\newcommand{\mdiv}{\textup{div}}

\newcommand{\wot}{W^{1,2}\left(\Omega\right)}
\newcommand{\tuj}{\tilde{u}_j}
\newcommand{\buj}{\bar{u}_j}
\newcommand{\brj}{\bar{\rho}_j}

\begin{document}
	
	\title[]{Exponential Crystal Relaxation Model With P-Laplacian}
	\author{Brock C. Price and Xiangsheng Xu}\thanks
	{Department of Mathematics and Statistics, Mississippi State
		University, Mississippi State, MS 39762.
		{\it Email}: bcp193@msstate.edu (Brock C. Price); xxu@math.msstate.edu (Xiangsheng Xu).}
	 \subjclass{}
	
	\begin{abstract} 
		In this article we prove the global existence of weak solutions to an initial boundary value problem with an exponential and p-Laplacian nonlinearity. The equation is a continuum limit of a family of kinetic Monte Carlo models of crystal surface relaxation. In our investigation we find a weak solution where the exponent in the equation, $-\Delta_p u$, can have a singular part in accordance with the Lebesgue Decomposition Theorem. The singular portion of $-\Delta_p u$ corresponds to where $-\Delta_p u = -\infty$, which leads it to have a canceling effect with the exponential nonlinearity. This effect has already been demonstrated for the case of a linear exponent $p=2$, and for the time independent problem. Our investigation reveals that we can exploit this same effect in the time dependent case with nonlinear exponent. We obtain a solution by first forming a sequence of approximate solutions and then passing to the limit. The key to our existence result lies in the observation that one can still obtain the precompactness of the term $e^{-\Delta_p u}$ despite a complete lack of estimates in the time direction. However, we must assume that $1<p\leq 2$.
	\end{abstract}	
	
	\maketitle

	\section{Introduction}

Let $\Omega$ be a bounded domain in $\Omega$ with Lipschitz boundary $\partial\Omega$, $T>0$, and $1<p\leq2$. In this article we consider the initial-boundary value problem,
\begin{align}
	\partial_t u - \Delta e^{-\Delta_p u} &= 0 \ \ \mbox{ in $\ot \equiv \Omega\times\left(0,T\right)$, } \label{a} \\
	\nabla u \cdot \nu = \nabla e^{-\Delta_p u} \cdot \nu & = 0 \ \ \mbox{ on $\Sigma_T \equiv \partial\Omega \times\left(0,T \right)$, } \label{b} \\
	u(x,0) &= u_0(x) \ \ \mbox{ in $\Omega\times\{0\}$. } \label{c} 
\end{align}

The equation \eqref{a} can be used to describe the relaxation of a crystal surface \cite{BCF,MW,KDM,YG,GLL,GLL2}. Below the roughing temperature, the surface of a crystal consists of steps and terraces. Atoms then detach from steps, diffuse across terraces, and then reattach at a new location \cite{BCF}.  In our context, $u(x,t)$ is the surface height. To get to our equation, we let $J$ be the adatom flux. We can then express the conservation of mass as,
\begin{equation}\label{com}
	\partial_t u + \mdiv J = 0. 
\end{equation}
Then, from Fick's Law \cite{MK}, we can express $J$ as,
\begin{equation}
	-M(\nabla u) \nabla \rho_s, 
\end{equation}
where $M$ is the mobility and $\rho_s$ is the local equilibrium density of adatoms. Using the Gibbs-Thomson relation we can write $\rho_s$ as,
\begin{equation}
	\rho_s = \rho_0 e^{\frac{\mu}{kT_0}}, 
\end{equation}
where $\mu$ represents the chemical potential, $\rho_0$ is the constant reference density, $k$ is the Boltzmann constant, and $T_0$ is temperature. We let $\Omega$ be the location of the steps, and the general surface energy $G(u)$ to be,
\begin{equation}
	G(u) = \frac{1}{p} \int_{\Omega} \left|\nabla u \right|^p dx, \ \ \mbox{ where $p\geq 1$. } 
\end{equation}
Then, the chemical potential is defined to be the change per atom in the surface energy, 
\begin{equation}
	\mu = \frac{\delta G}{\delta u} = - \Delta_p u. 
\end{equation}
In the diffusion-limited regime the dynamics are primarily due to the diffusion across terraces, and so we can take $M = 1$. We refer the reader to \cite{YG,XX2} for the case where $M$ is not a constant. Upon setting the other physical constants to be one our equation \eqref{com} reduces to,
\begin{equation}
	\partial_t u - \Delta e^{-\Delta_p u} = 0. 
\end{equation}

We now make the following definition of a weak solution to equation \eqref{a}-\eqref{c};
\begin{defn}\label{defn.of.weak.sol.}
	We say that a pair of functions $\left(\rho, u \right)$ is a weak solution of \eqref{a}-\eqref{c} if the following conditions hold:
	
	(D1) The function $u$ satisfies, $u \in L^{p}\left(0,T; W^{1,p}\left(\Omega\right)\right)$, $\partial_t u \in L^1\left(0,T;\left(W^{1,\infty}\left(\Omega\right)\right)^*\right)$, where $\left(W^{1,\infty}\left(\Omega\right)\right)^{*}$ denotes the dual space of $W^{1,\infty}\left(\Omega\right)$. 
	
	(D2) The function $\rho$ satisfies, $\rho\geq 0,\ \sqrt{\rho} \in L^2\left(0,T; \wot\right)$. 
	
	(D3) $\Delta_p u \in  \mathbb{M}\left(\overline{\ot}\right) \cap L^p\left(0,T; \left(W^{1,p}\left(\Omega\right)\right)^{*}\right)$, where $\mathbb{M}\left(\overline{\ot}\right)$ is the space of signed Radon measures on $\overline{\ot}$.
	
	(D4) The term $-\Delta_p u$ has a decomposition, from the Lebesgue Decomposition Theorem, 
	\begin{equation}
		-\Delta_p u = g_a + \nu_s, 
	\end{equation}
	where $g_a$ is the absolutely continuous part, and $\nu_s$ is the singular part, and the support of $\nu_s$, $A_0$, has Lebesgue measure zero. Then we have,
	\begin{equation}
		\rho = e^{g_a}, \ \ \mbox{ a.e. on $\ot$. } 
	\end{equation}

	(D5) The functions $\left(u,\rho\right)$ satisfy 
	the integral equation,
	\begin{equation}
		-\int_{\ot} u \partial_t \varphi dxdt + \int_{\ot} \nabla \rho \cdot \nabla \varphi dxdt = \int_{\Omega} u_0(x)\varphi(x,0) dx, \ \ \mbox{ for all smooth $\varphi$ that vanish for $t=T$. } \label{defnthree} 
	\end{equation}

(D6) There holds
\begin{equation}\label{rbx}
	\int_{0}^{T} \langle -\Delta_p u, \xi \rangle dt = \int_{\ot} \left|\nabla u \right|^{p-2} \nabla u \cdot \nabla \xi dxdt \ \ \mbox{ for all smooth $\xi$,}
\end{equation}
where $\langle \cdot, \cdot \rangle$ is the duality pairing between $\left(W^{1,p}\left(\Omega\right)\right)^{*}$ and $W^{1,p}\left(\Omega\right)$. 
\end{defn}

The initial condition and the second boundary condition in \eqref{b} have been incorporated into the integral equation \eqref{defnthree}, while the first boundary condition in \eqref{b} is implied by \eqref{rbx}.

We would like to remark that the singular part in $-\Delta_p u$ cannot be ruled out. (See for example \cite{LX,YG,GLL}.) Even when the exponent is linear, and the space dimension is only $1$, there is a counterexample in \cite{LX}, demonstrating that there can indeed be a singular part. In the case of a linear exponent, to remove the singularity one must impose some sort of smallness condition on the initial data \cite{LS,PX,BM}. We also see
 the gradient flow theory being employed for the existence assertion \cite{YG,GLL}. Our problem here does not seem to have a gradient flow structure. Various time-independent cases have been investigated in \cite{XX,XX2,XX3}.

Our main result is then the following, 
\begin{thm}[Main Theorem]\label{mainthm}
	Suppose that $\Omega$ is a bounded domain in $\mathbb{R}^N$ with Lipschitz boundary, and suppose that $u_0(x) \in W^{1,p}\left(\Omega\right)\cap L^2\left(\Omega\right)$. Then there exists a weak solution to \eqref{a}-\eqref{c} in the sense of Definition \ref{defn.of.weak.sol.}. 
\end{thm}

A key assumption in our result is that $1<p\leq 2$. Even though $p>2$ should make analysis easier, our proof is not valid in this case. To see why we require that $1<p\leq 2$, we will go through the a-priori estimates which lead to our result. That is to say,  we will assume that we have a classical solution $u$ to equation \eqref{a} and derive estimates that $u$ must satisfy. To begin we first set $\rho = e^{-\Delta_p u}$, which gives us the system of equations,

\begin{align}
	\partial_t u - \Delta \rho &= 0, \ \ \mbox{ in $\ot$, } \label{one} \\
	-\Delta_p u & = \ln(\rho), \ \ \mbox{ in $\ot$, } \label{two}\\
	\nabla \rho\cdot \nu = \nabla u \cdot \nu &= 0, \ \ \mbox{ on $\partial \Omega$, }  \\
	u(x,0) & = u_0(x), \ \ \mbox{ on $\Omega$. } 
\end{align}

We then use $\ln(\rho)$ as a test function in \eqref{one} to get,
\begin{equation}
	\frac{1}{p} \frac{d}{dt} \int_{\Omega} \left|\nabla u \right|^p dx + \int_{\Omega} \left|\nabla \sqrt{\rho} \right|^2 dx = 0. 
\end{equation}
Then, integrate w.r.t. $t$ to get the first estimate,
\begin{equation}\label{apriorione}
	\sup_{0\leq t \leq T} \int_{\Omega} \left|\nabla u \right|^p dx + \int_{\ot} \left|\nabla \sqrt{\rho} \right|^2 dxdt \leq c \int_{\Omega} \left| \nabla u_0(x) \right|^p dx. 
\end{equation}

Next, we use $\sqrt{\rho}$ as a test function in \eqref{two} to achieve,
\begin{align*}
	\int_{\Omega} \sqrt{\rho} \ln(\rho) dx &= \int_{\Omega} \left|\nabla u \right|^{p-2} \nabla u \cdot \nabla \sqrt{\rho} dx \\
	& \leq \int_{\Omega} \left| \nabla u \right|^{p-1} \left| \nabla \sqrt{\rho} \right| dx \\
	& \leq \left\|\nabla u \right\|^{p-1}_p \left\| \nabla \sqrt{\rho} \right\|_p \\
	& \leq c \left\|\nabla u \right\|^{p-1}_p \left\| \nabla \sqrt{\rho} \right\|_2.
\end{align*}
The last step is due to the assumption $p\leq 2$. In view of \eqref{apriorione}, we would encounter an unbridgeable gap if we had $p>2$. To continue, for the left-hand side of the preceding inequality we have,  
\begin{equation*}
	\int_{\Omega} \sqrt{\rho} \ln(\rho) dx    \geq 2 \int_{\Omega}\left( \sqrt{\rho} - 1\right) dx . 
\end{equation*}
Plugging this back in we find,
\begin{equation}
	\int_{\Omega} \sqrt{\rho} dx \leq c \left\|\nabla u \right\|^{p-1}_p \left\| \nabla \sqrt{\rho} \right\|_2 + c. 
\end{equation}
Next, we use the Sobolev inequality,
\begin{equation}
	\left\| u \right\|_{q^*} \leq c \left( \left\| \nabla u \right\|_q + \left\| u \right\|_1 \right), 
\end{equation}
with $q=2$ to find,
\begin{align}
	\int_{\Omega} \left| \sqrt{\rho} \right|^2 dx & \leq c\left( \int_{\Omega} \left| \sqrt{\rho} \right|^{\frac{2N}{N-2}} dx \right)^{\frac{N-2}{N}} \\
	& \leq c \int_{\Omega} \left| \nabla \sqrt{\rho} \right|^2 dx + c \left(\int_{\Omega} \sqrt{\rho} dx \right)^2 \\
	& \leq c \int_{\Omega} \left| \nabla \sqrt{\rho} \right|^2 dx + c \left\|\nabla u \right\|^{2(p-1)}_p \int_{\Omega} \left|\nabla \sqrt{ \rho} \right|^2 dx + c,
\end{align}
Integrating this inequality with respect to $t$ we then find our second a-priori estimate,
\begin{equation}\label{aprioritwo}
	\int_{\ot} \left| \sqrt{\rho} \right|^2 dx dt \leq c \int_{\ot} \left| \nabla \sqrt{\rho} \right|^2 dx dt + c \left(\sup_{0\leq t \leq T}\left\|\nabla u \right\|_p \right)^{2(p-1)} \int_{\ot} \left| \nabla \sqrt{\rho} \right|^2 dx dt + c
\end{equation}

Next, we integrate \eqref{two} over $\Omega$ to find,
\begin{equation}
	\int_{\Omega} \ln(\rho) dx = 0. 
\end{equation}
Using this we then get our third estimate,
\begin{align}
	\int_{\ot} \left|\ln(\rho)\right| dxdt & = \int_{\ot} \ln^{+}(\rho) dxdt + \int_{\ot} \ln^{-}(\rho) dxdt \\
	& = 2 \int_{\ot} \ln^{+}(\rho) dxdt - \int_{\ot} \ln(\rho) dxdt \\
	& \leq 2 \int_{\ot} \sqrt{\rho} dxdt \\
	& \leq c. 
\end{align}

Finally we integrate \eqref{one} over $\Omega$ to find,
\begin{equation}
	\frac{d}{dt} \int_{\Omega} u dx = 0. 
\end{equation}
Subsequently we have,
\begin{equation}
	\int_{\Omega} u(x,t) dx = \int_{\Omega} u_0(x) dx. 
\end{equation}

Once again, our second estimate, \eqref{aprioritwo}, is only possible due to the constraint that $1<p\leq2$. This restriction will also come into play in a similar manner to demonstrate the precompactness of our approximating sequence for the term $\sqrt{\rho}$. This is accomplished despite the fact that we have no estimates in the time direction for $\rho$. 

\section{Preliminary Results}

In this section we will collect a number of key lemmas that we need. With our first lemma we collect several elementary inequalities that we will be using,
\begin{lemma}\label{inequalities}
	For $x,y \in \mathbb{R}^N$ and $a,b \in \mathbb{R}^{+}$, we have the inequalities,
	\begin{align}
		\left|x\right|^{p-2} x \cdot\left(x-y\right) \geq \frac{1}{p} \left(\left|x\right|^p - \left|y\right|^p \right); \\
		ab \leq \varepsilon a^p + \frac{1}{\varepsilon^{\frac{q}{p}}}b^q \ \ \mbox{ for $\varepsilon > 0$ and $p,q >1$ with $\frac{1}{p} + \frac{1}{q} = 1$, } \\
		\left(\sqrt{a}-\sqrt{b}\right)\left(\ln(a) - \ln(b) \right) \geq 4\left( a^{\frac{1}{4}} - b^{\frac{1}{4}} \right)^2. 
	\end{align}
	\end{lemma}
The proof of the last inequality is contained in \cite{LX}.
\begin{lemma}\label{odenone}
	For $1<p\leq 2$, and $x,y \in \mathbb{R}^N$, 
	\begin{equation}
		\left(1+\left|x\right|^2+\left|y\right|^2\right)^{\frac{2-p}{2}}\left(\left(\left|x\right|^{p-2} x - \left|y\right|^{p-2} y\right)\cdot\left(x-y\right)\right) \geq (p-1)\left|x-y\right|^2. 
	\end{equation}
\end{lemma}
The proof of this lemma can be found in \cite{JOden}. 

Our existence assertion will come as a consequence of the following fixed point theorem which is often called the Leray-Schauder Theorem. (\cite{GT}, chap. 11)
\begin{lemma}
	Let $T$ be a compact mapping of a Banach space $\mathbb{B}$ into itself, and suppose there exists a constant $M$ so that,
	\begin{equation}
		\left\| x \right\|_{\mathbb{B}} < M
	\end{equation}
	for all $x\in\mathbb{B}$ and $\sigma \in \left[0,1\right]$ which satisfy $x = \sigma T(x)$. Then $T$ has a fixed point. 
\end{lemma}

A critical point in the proof of our main theoremm comes from being able to go from weak convergence to strong convergence for our sequence of approximate solutions. This will come as a consequence of the Lions-Aubin lemma \cite{S}, which is the next result.

\begin{lemma}[Lions-Aubin]\label{la}
	Let $X_0, X$ and $X_1$ be three Banach spaces with $X_0 \subseteq X \subseteq X_1$. Suppose that $X_0$ is compactly embedded in $X$ and that $X$ is continuously embedded in $X_1$. For $1 \leq p, q \leq \infty$, let
	\begin{equation*}
		W=\{u\in L^{p}([0,T];X_{0}): \partial_t u\in L^{q}([0,T];X_{1})\}.
	\end{equation*}
	Then:
	\begin{enumerate}
		\item[\textup{(i)}] If $p  < \infty$, then the embedding of W into $L^p([0, T]; X)$ is compact.
		\item[\textup{(ii)}] If $p  = \infty$ and $q  >  1$, then the embedding of W into $C([0, T]; X)$ is compact. 
	\end{enumerate}
\end{lemma}

Our next lemma is an existence result for the problem,
\begin{align}
	-\Delta \rho + \tau\ln(\rho) &= f, \ \ \mbox{ in $\Omega$ } \label{rone} \\
	\nabla \rho \cdot \nu &= 0, \ \ \mbox{ on $\partial\Omega$. } \label{rtwo} 
\end{align}
A weak solution to \eqref{rone}-\eqref{rtwo} is a function $\rho \in \wot$ with,
\begin{equation}
	\rho > 0, \ \ \mbox{ a.e. on $\Omega$, } \ \ \ln(\rho) \in L^2\left(\Omega\right),
\end{equation}
which also satisfies the integral equation,
\begin{equation}
	\int_{\Omega} \nabla \rho \cdot \nabla \varphi dx + \tau \int_{\Omega} \ln(\rho) \varphi dx = \int_{\Omega} f\varphi dx \ \ \mbox{ for all $\varphi\in\wot$.} 
\end{equation}
Then we have the following existence theorem, (\cite{XX}, lemma 3.3)
\begin{lemma}\label{rhologexistence}
	Suppose $f\in L^2\left(\Omega\right)$, and $\tau > 0$. Then there is a unique weak solution to \eqref{rone}-\eqref{rtwo}. Furthermore, we have the estimate,
	\begin{equation}
		\left\|\tau\ln(\rho)\right\|_{2,\Omega} \leq \left\|f \right\|_{2,\Omega}. 
	\end{equation}
\end{lemma}

\section{Approximate Problems}

In this section we investigate our approximate problems. We largely follow \cite{XX}. The approximation scheme is based on an implicit discretization in time. Let $\tau \in \left(0,1\right)$, and $v\in L^2\left(\Omega\right)$ be given, then our approximate problems are,
\begin{align}
	-\Delta \rho + \tau \ln\left(\rho\right) &= -\frac{u-v}{\tau}, \ \ \mbox{ in $\Omega$, } \label{approx.one} \\
	- \Delta_p u + \tau u &= \ln\left(\rho\right), \ \ \mbox{ in $\Omega$, } \label{approx.two} \\
	\nabla u \cdot \nu = \nabla \rho \cdot \nu & = 0, \ \ \mbox{ on $\partial\Omega$. } \label{approx.three}
\end{align}
In this section our aim is to establish the following existing theorem for \eqref{approx.one}-\eqref{approx.two},
\begin{prop}\label{approx.problem.exist}
	There exists a weak solution $\left(u,\rho\right)$ with $u\in W^{1,p}\left(\Omega\right)\cap L^2\left(\Omega\right)$, and $\rho \in \wot$, so that, $\rho > 0$ a.e. on $\Omega$, and $\ln(\rho) \in L^2\left(\Omega\right)$. 
\end{prop}

The existence of solutions to our approximate problems will be established via a weak convergence method. Before we prove Proposition \ref{approx.problem.exist}, we first need the following intermediate proposition;
\begin{prop}\label{approx.delta.problem.exist}
	For each $\delta \in \left(0,1\right)$, and $v\in L^2\left(\Omega\right)$ there is a unique solution $\left(\rho, u \right)$ in the space $\wot\times\wot$, to the problem,
	\begin{align}
		-\Delta \rho + \tau \ln\left(\rho\right) & = -\frac{u-v}{\tau}, \ \ \mbox{ in $\Omega$, } \label{delta.one} \\
		-\Delta_p u - \delta\Delta u + \tau u & = \ln\left(\rho\right), \ \ \mbox{ in $\Omega$, } \label{delta.two} \\
		\nabla u \cdot \nu = \nabla \rho \cdot \nu & = 0, \ \ \mbox{ on $\partial\Omega$. } \label{delta.three}
	\end{align}
\end{prop}
\begin{proof}
	For the proof of existence we will use the Leray-Schauder Theorem. To do so, we define an Operator $\mathbb{T}:L^2\left(\Omega\right) \rightarrow L^2\left(\Omega\right)$ in the following manner; Given $w \in L^2\left(\Omega\right)$, we first define $\rho \in \wot$ to be the unique solution given by lemma \ref{rhologexistence} to,
	\begin{align}
		-\Delta \rho + \tau \ln(\rho) &= -\frac{w-v}{\tau}, \ \ \mbox{ in $\Omega$, } \\
		\nabla \rho \cdot \nu &= 0 \ \ \mbox{ on $\partial\Omega$. }
	\end{align}
	 
	 Next, we define $u \in\wot$ to be the unique solution to the equations,
	 \begin{align}
	 	-\Delta_p u - \delta \Delta u + \tau u &= \ln(\rho), \ \ \mbox{ in $\Omega$, } \\
	 	\nabla u \cdot \nu &= 0, \ \ \mbox{ on $\partial\Omega$. }
	 \end{align}
 	We then define $u = \mathbb{T}(w)$. From the uniqueness of the solutions $\rho$ and $u$ we can see that $\mathbb{T}$ is well-defined. To make use of the Leray-Schauder Theorem we need to show that $\mathbb{T}$ is continuous, maps bounded sets into precompact ones, and that there exists a constant $M$ so that,
 	\begin{equation}
 		\left\| u \right\|_{2,\Omega} \leq M,
 	\end{equation}
 	for all $u\in L^2\left(\Omega\right)$ and $\sigma \in \left[0,1\right]$ so that $ u = \sigma \mathbb{T}(u)$. We start by demonstrating the continuity of $\mathbb{T}$. Suppose that $\{w_n\}$ is a sequence of functions in $L^2\left(\Omega\right)$ so that, $w_n \rightarrow w$ in $L^2\left(\Omega\right)$ for some function $w\in L^2\left(\Omega\right)$. Next, set $u_n = \mathbb{T}(w_n)$ for all $n=1,2,3,...$ That is for each $n$ we have a pair $\left(\rho_n, u_n\right)$, which satisfy the equations,
 	\begin{align}
 		- \Delta \rho_n + \tau\ln(\rho_n) &= - \frac{w_n - v}{\tau}, \ \ \mbox{ in $\Omega$, } \label{rho.n} \\
 		-\Delta_p u_n -\delta\Delta u_n + \tau u_n &= \ln(\rho_n), \ \ \mbox{ in $\Omega$, } \label{u.n} \\
 		\nabla u_n \cdot \nu = \nabla \rho_n \cdot nu &= 0, \ \ \mbox{ on $\partial\Omega$. } \label{boundary.n}
 	\end{align}
 	By lemma \ref{rhologexistence} we then have the estimate,
 	\begin{equation}\label{logrho.n.est}
 		\left\|\ln(\rho_n) \right\|_{2,\Omega} \leq \frac{1}{\tau} \left\|\frac{w_n -v}{\tau} \right\|_{2,\Omega} \leq c(\tau).
 	\end{equation}
 	Next, we use $u_n$ as a test function in \eqref{u.n} to get,
 	\begin{align}
 		\int_{\Omega} \left|\nabla u_n \right|^p dx + \delta\int_{\Omega} \left|\nabla u_n \right|^2 dx + \tau\int_{\Omega} u^2_n dx &\leq \int_{\Omega} \ln(\rho_n) u_n dx \\
 		& \leq \varepsilon\int_{\Omega} u^2_n dx + \frac{1}{\varepsilon} \int_{\Omega} \ln^2(\rho) dx. 
 	\end{align}
 	By choosing $\varepsilon$ to be smaller than $\tau$ we get,
 	\begin{equation}\label{u.n.est}
 		\int_{\Omega} \left|\nabla u_n \right|^2 dx + \int_{\Omega} u^2_n dx \leq c(\delta,\tau). 
 	\end{equation}
 	Next, we use $\left(\rho_n - 1\right)$ as a test function in \eqref{rho.n} to obtain,
 	\begin{align}
 		\int_{\Omega} \left|\nabla\rho_n\right|^2 dx + &\tau\int_{\Omega} \ln(\rho_n)\left(\rho_n - 1\right) dx  = -\frac{1}{\tau}\int_{\Omega} \left(w_n -v\right)\left(\rho_n -1\right) dx \\
 		& \leq \frac{1}{\tau}\left\|w_n - v\right\|_{\frac{2N}{N+2},\Omega} \left\|\rho_n - 1 \right\|_{\frac{2N}{N-2},\Omega} \\
 		& \leq c\left\|w_n - v\right\|_{2,\Omega}\left( \left\|\nabla \rho_n \right\|_{2,\Omega} + \left\|\rho_n - 1 \right\|_{1,\Omega}\right) \\
 		& \leq \varepsilon \left\|\nabla\rho_n\right\|^2_{2,\Omega} + c(\varepsilon)\left\|w_n - v\right\|_{2,\Omega} + c\left\|w_n - v\right\|_{2,\Omega} \left\|\rho_n - 1\right\|_{1,\Omega}. 
 	\end{align}
 	By choosing $\varepsilon$ suitably small we obtain,
 	\begin{equation}\label{e.one}
 		\int_{\Omega} \left|\nabla \rho_n \right|^2 dx + \tau\int_{\Omega}\ln(\rho_n)\left(\rho_n -1\right) dx \leq c\left\|\rho_n -1 \right\|_{1,\Omega} + c. 
 	\end{equation}
 	Now, suppose that $M$ is any positive number. Then we use the inequality \eqref{e.one} to find,
 	\begin{align}
 		\int_{\Omega} \left|\rho_n - 1\right| dx & = \int_{\{\rho_n > M\}} \left|\rho_n - 1\right| dx + \int_{\{\rho_n \leq M\}} \left|\rho_n - 1\right| dx \\
 		& \leq \frac{1}{\ln(M)} \int_{\Omega} \ln(\rho_n)\left(\rho_n - 1\right) dx + c(M) \\
 		& \leq \frac{c}{\ln(M)} \int_{\Omega} \left|\rho_n - 1\right| dx + c(M). 
 	\end{align}
 	By choosing a sufficiently large $M$ we then find,
 	\begin{equation}
 		\int_{\Omega} \left|\rho_n - 1 \right| dx \leq c. 
 	\end{equation}
 	Using this is \eqref{e.one} then yields,
 	\begin{equation}
 		\int_{\Omega} \left|\nabla \rho_n \right|^2 dx \leq c. 
 	\end{equation}
 	Thus by the Sobolev Inequality, we can conclude that $\rho_n$ is bounded in $\wot$. Taking this and \eqref{u.n.est} into consideration we can conclude that (at least for a subsequence) there are functions $\left(u,\rho\right)$ so that,
 	\begin{align}
 		u_n \rightarrow u & \ \mbox{ weakly in $\wot$, and strongly in $L^2\left(\Omega\right)$, } \label{u.n.convergence}\\
 		\rho_n \rightarrow \rho & \ \mbox{ weakly in $\wot$, and strongly in $L^2\left(\Omega\right)$. } 
 	\end{align}
 	It then follows from \eqref{logrho.n.est} that,
 	\begin{equation}
 		\ln(\rho_n) \rightarrow \ln(\rho) \ \ \mbox{ strongly in $L^2(\Omega)$. }
 	\end{equation}
 	Now, by \eqref{u.n} we have,
 	\begin{equation}
 		-\mdiv\left[ \left|\nabla u_n\right|^{p-2}\nabla u_n - \left|\nabla u_m \right|^{p-2} \nabla u_m \right] - \delta\Delta \left(u_n -u_m\right) + \tau\left(u_n-u_m\right) = \ln(\rho_n) - \ln(\rho_m), \ \ \mbox{ in $\Omega$. } 
 	\end{equation}
 	We use the function $\left(u_n - u_m\right)$ as a test function in the above equation to then obtain,
 	\begin{align}\label{unum}
 		\int_{\Omega}& \left( \left|\nabla u_n\right|^{p-2}\nabla u_n - \left|\nabla u_m \right|^{p-2} \nabla u_m \right) \cdot\nabla\left(u_n-u_m\right) dx \\
 		&+\delta\int_{\Omega} \left|\nabla\left(u_n-u_m\right) \right|^2dx +\tau\int_{\Omega} \left(u_n-u_m\right)^2dx \\
 		& \leq \int_{\Omega}\left(\ln(\rho_n)-\ln(\rho_m)\right)\left(u_n-u_m\right) dx 
 	\end{align}
 	Then, we use lemma \ref{odenone} to find that,
 	\begin{align}
 		\int_{\Omega}& \left| \nabla\left(u_n-u_m\right) \right|^p dx \nonumber \\
 		& = \int_{\Omega} \left(1+\left|\nabla u_n\right|^2 + \left|\nabla u_m\right|^2 \right)^{\frac{p(2-p)}{4}} \frac{\left| \nabla\left(u_n-u_m\right) \right|^p }{\left(1+\left|\nabla u_n\right|^2 + \left|\nabla u_m\right|^2 \right)^{\frac{p(2-p)}{4}}} dx \nonumber \\
 		& \leq \left( \int_{\Omega} \left(1+\left|\nabla u_n\right|^2 + \left|\nabla u_m \right|^2\right)\left|\nabla\left(u_n - u_m \right) \right|^2 dx \right)^{\frac{p}{2}} \nonumber \\
 		&\cdot\left(\int_{\Omega} \left(1+\left|\nabla u_n \right|^2 + \left|\nabla u_m \right|^2 \right)^{\frac{p}{2}} dx \right)^{1-\frac{p}{2}} \nonumber \\
 		& \leq c \int_{\Omega} \left( \left|\nabla u_n\right|^{p-2}\nabla u_n - \left|\nabla u_m \right|^{p-2} \nabla u_m \right) \cdot\nabla\left(u_n-u_m\right) dx \label{redundant}
 	\end{align}
 	This estimate is not strictly necessary at this point. However, when it comes time to take the limit $\delta \rightarrow 0$ we will make use of it. Using \eqref{redundant} in \eqref{unum} we find,
 	\begin{align}
 		\int_{\Omega} \left| \nabla\left(u_n-u_m\right) \right|^p dx + \delta \int_{\Omega} \left| \nabla \left(u_n - u_m \right)\right|^2 dx \\
 		+ \tau \int_{\Omega} \left(u_n - u_m \right)^2 dx \leq \int_{\Omega} \left(\ln(\rho_n)-\ln(\rho_m)\right)\left(u_n-u_m\right) dx 
 	\end{align}
 	Subsequently, we can conclude that $\{u_n\}$ is precompact in $\wot$. At this point we can now pass to the limit in \eqref{rho.n}-\eqref{boundary.n}. Then, the uniqueness for solutions to \eqref{rho.n}-\eqref{boundary.n} implies that the whole sequence converges. Therefore we can conclude that $B$ is continuous. 
 	
 	The fact that $\mathbb{T}$ maps bounded sets into precompact ones is already demonstrated on account of \eqref{u.n.convergence}. The only piece remaining to be able to use the Leray-Schauder Theorem is to show the existence of a constant $M$ so that,
 	\begin{equation}
 		\left\|u \right\|_{2,\Omega} \leq M
 	\end{equation}
 	for all $ u \in L^2\left(\Omega\right)$ and $\sigma \in \left[0,1\right]$ satisfying $u=\sigma \mathbb{T}(u)$. This equation is equivalent to the boundary value problem,
 	\begin{align}
 		-\Delta \rho + \tau\ln(\rho) &= -\frac{u-v}{\tau} , \ \ \mbox{ in $\Omega$, } \label{rhoM} \\
 		-\mdiv\left[\left|\nabla\frac{u}{\sigma}\right|^{p-2} \nabla u \right] - \delta \Delta u + \tau u &= \sigma\ln(\rho), \ \ \mbox{ in $\Omega$, } \label{uM} \\
 		\nabla \rho \cdot\nu = \nabla u \cdot \nu & = 0, \ \ \mbox{ on $\partial\Omega$. } 
 	\end{align}
 	We use $\ln(\rho)$ as a test function in \eqref{rhoM} to find,
 	\begin{equation}\label{Mone}
 		\int_{\Omega} \left|\nabla\sqrt{\rho}\right|^2 dx + \tau \int_{\Omega} \ln^2(\rho) dx  = - \frac{1}{\tau}\int_{\Omega}\left(u-v\right) \ln(\rho) dx. 
 	\end{equation}
 	Next, we use $u$ as a test function in \eqref{uM} to obtain,
 	\begin{equation}\label{Mtwo}
 		\int_{\Omega} \ln(\rho) u dx = \int_{\Omega} \left|\nabla\frac{u}{\sigma}\right|^{p-2}\left|\nabla u \right|^2 dx + \delta \int_{\Omega} \left|\nabla u \right|^2 dx + \tau \int_{\Omega} u^2 dx \geq 0. 
 	\end{equation}
 	Incorporating this into \eqref{Mone} we obtain,
 	\begin{align}
 		\int_{\Omega} \left|\nabla\sqrt{\rho}\right|^2 dx + \tau \int_{\Omega} \ln^2(\rho) dx & \leq \frac{1}{\tau} \int_{\Omega} v \ln(\rho) dx \\
 		& \leq \varepsilon \int_{\Omega} \ln^2(\rho) dx + c(\varepsilon) \int_{\Omega} v^2 dx.
 	\end{align}
 	Upon choosing $\varepsilon$ sufficiently small we find,
 	\begin{equation}
 		\int_{\Omega} \left|\nabla\sqrt{\rho}\right|^2 dx + \tau \int_{\Omega} \ln^2(\rho) dx \leq c\int_{\Omega} v^2 dx \leq c. 
 	\end{equation}
 	Then, from equation \eqref{Mtwo} we find that,
 	\begin{align}
 		\tau \int_{\Omega} u^2 dx &\leq \int_{\Omega} \ln(\rho) u dx \\
 		& \leq \varepsilon \int_{\Omega} u^2 dx + c(\varepsilon) \ln^2(\rho) dx. 
 	\end{align}
 	Again choosing a sufficiently small $\varepsilon$ gives us,
 	\begin{equation}
 		\int_{\Omega} u^2 dx \leq c \int_{\Omega}\ln^2(\rho) dx \leq c. 
 	\end{equation}
 	At this point we are now ready to call upon the Leray-Schauder Theorem to conclude that $\mathbb{T}$ has a fixed point. Clearly a fixed point of $\mathbb{T}$ is a solution to \eqref{approx.one}-\eqref{approx.three}. The proof is complete. 
\end{proof}

We can now continue on to the proof of Proposition \ref{approx.problem.exist}. To do so we need to show that we can take $\delta \rightarrow 0$ in equations \eqref{delta.one}-\eqref{delta.three}. All of the key steps to being able to do so are already contained in the proof of Proposition \ref{approx.delta.problem.exist}. The difference being, we no longer have that $u_{\delta}$ converges strongly in $\wot$. However, on account of \eqref{redundant}, we are able to conclude that $u_{\delta}$ converges strongly in $W^{1,p}\left(\Omega\right)$. This is enough for us to be able to pass to the limit in \eqref{delta.one}-\eqref{delta.two}. The proof of Proposition \ref{approx.problem.exist} is complete. 

\section{Proof of the Main Theorem} 

In this section we will prove our Main Theorem. To prove our main theorem, we will first obtain estimates which are the discrete analogs of the a-priori estimates in the introduction. We then show that these estimates are enough for us to be able to pass to the limit. 

Let $T > 0$ be given. For all $j\in\{1,2,3,...\}$ we partition the time interval $\left[0,T\right]$ into $j$ equally sized sub-intervals. Then, we put,
\begin{equation}
	\tau = \frac{T}{j}.
\end{equation}
Now, let $u_0$ be a given function in $W^{1,p}\left(\Omega\right)\cap L^2\left(\Omega\right)$. Then for $k=1,2,...,j$, we use Proposition \ref{approx.problem.exist} to recursively solve the system of equations,
\begin{align}
	\frac{u_k - u_{k-1}}{\tau} - \Delta \rho_k + \tau\ln(\rho_k) &= 0, \ \ \mbox{in $\Omega$, } \label{mthmone} \\
	-\Delta_p u_k + \tau u_k &= \ln(\rho_k), \ \ \mbox{ in $\Omega$, } \label{mthmtwo} \\
	\nabla \rho_k \cdot \nu = \nabla u_k \cdot \nu &= 0, \ \ \mbox{ on $\partial\Omega$. } \label{mthmthree}
\end{align}
We introduce the functions,
\begin{align}
	\tuj &=  \frac{t-t_{k-1}}{\tau} u_k(x) + \left(1-\frac{t-t_{k-1}}{\tau} \right) u_{k-1}(x),  \ \ \mbox{ for $x\in\Omega$ and $t\in\left(t_{k-1}, t_k\right]$, }  \\
	\buj &=  u_k(x), \ \ \mbox{ for $x\in\Omega$ and $t\in\left(t_{k-1}, t_k\right]$, }   \\
	\brj &=  \rho_k(x) \ \ \mbox{ for $x\in\Omega$ and $t\in\left(t_{k-1}, t_k\right]$. }
\end{align}
Where $t_k = k\tau$. 

With these functions we can write the system \eqref{mthmone}-\eqref{mthmthree} as,
\begin{align}
	\partial_t \tuj -\Delta \brj + \tau \ln(\brj) & = 0, \ \ \mbox{ in $\ot$, } \label{mthmfour} \\
	-\Delta_p \buj + \tau \buj &= \ln(\brj), \ \ \mbox{ in $\ot$. } \label{mthmfive} 
\end{align}

We proceed to derive estimates for the sequences, $\{\tuj,\buj,\brj\}$ which are independent of $\tau$. First, we have the discrete analog of \eqref{apriorione}, 
\begin{lemma}\label{lemma4.1}
	We have the estimate,
	\begin{equation}
		\begin{aligned}
			\frac{1}{p}& \max_{0 \leq t \leq T} \int_{\Omega} \left|\nabla \buj\right|^p dx + 4 \int_{\ot} \left|\nabla\sqrt{\brj}\right|^2 dxdt \\
			& + \frac{\tau}{2} \max_{0 \leq t \leq T} \int_{\Omega} \buj^2 dx + \tau \int_{\ot} \ln^2(\brj) dxdt \\
			& \leq \frac{1}{p} \int_{\Omega} \left|\nabla u_0(x)\right|^p dx + \frac{\tau}{2} \int_{\Omega} u^2_0(x) dx. 
		\end{aligned}
	\end{equation}
\end{lemma}
\begin{proof}
	We first multiply through \eqref{mthmone} by $\ln(\rho_k)$ and integrate the resulting equation over $\Omega$ to obtain,
	\begin{equation}
		\int_{\Omega} \left(\frac{u_k - u_{k-1}}{\tau}\right) \ln(\rho_k) dx + 4\int_{\Omega} \left|\nabla \sqrt{\rho_k} \right|^2 dx + \tau\int_{\Omega} \ln^2(\rho_k) dx = 0. \label{4.1.1}
	\end{equation}
	Next, we multiply through equation \eqref{mthmtwo} by $\left(u_k - u_{k-1}\right)$ to find,
	\begin{equation}
		\begin{aligned}
			\int_{\Omega} \left(u_k - u_{k-1}\right)\ln(\rho_k) dx & = \tau \int_{\Omega} u_k\left(u_k - u_{k-1}\right) dx + \int_{\Omega} \left|\nabla u_k \right|^{p-2} \nabla u_k \cdot \nabla \left(u_k - u_{k-1}\right) dx \\
			& \geq \frac{\tau}{2}\int_{\Omega} \left( \left|u_k \right|^2 - \left|u_{k-1}\right|^2 \right) dx + \frac{1}{p} \int_{\Omega} \left(\left|\nabla u_k\right|^p + \left|\nabla u_{k-1} \right|^p \right) dx
		\end{aligned}
	\end{equation}
	Upon substituting this back into \eqref{4.1.1} we have,
	\begin{equation}
		\frac{1}{p\tau} \int_{\Omega} \left(\left|\nabla u_k\right|^p + \left|\nabla u_{k-1} \right|^p \right) dx + \frac{1}{2} \int_{\Omega}\left( \left|u_k \right|^2 - \left|u_{k-1}\right|^2 \right) dx  + 4\int_{\Omega} \left|\nabla \sqrt{\rho_k} \right|^2 dx + \tau\int_{\Omega} \ln^2(\rho_k) dx \leq 0.
	\end{equation}
	Multiplying through this inequality by $\tau$ and then summing the resulting equation over $k$ gives us the lemma. 
\end{proof}

Next, we have the discrete version of \eqref{aprioritwo},
\begin{lemma}\label{lemma4.2}
	There exists a constant $c$ which is independent of $\tau$ so that,
	\begin{equation}
		\int_{\ot} \left|\sqrt{\brj}\right|^2 dxdt \leq c. 
	\end{equation}
\end{lemma}
\begin{proof}
	We use $\sqrt{\rho_k}$ as a test function in \eqref{mthmtwo} fo find,
	\begin{equation}\label{4.2.1}
		\begin{aligned}
		\int_{\Omega} \sqrt{\rho_k} \ln(\rho_k) dx &= \tau \int_{\Omega} u_k \sqrt{\rho_k} dx + \int_{\Omega} \left|\nabla u_k \right|^{p-2}\nabla u_k \cdot \nabla \sqrt{\rho_k} dx \\
		& \leq \tau\int_{\Omega}u_k \sqrt{\rho_k} dx + \left\|\nabla u_k \right\|^{(p-1)}_{p,\Omega} \left\|\nabla \sqrt{\rho_k} \right\|_{2,\Omega} 
		\end{aligned}
	\end{equation}
Then, we have,
\begin{equation}
	\int_{\Omega} \sqrt{\rho_k} \ln(\rho_k) dx \geq 2\int_{\Omega} \left(\sqrt{\rho_k} -1\right) dx. 
\end{equation}
Putting this back into \eqref{4.2.1} gives us,
\begin{equation}
	\int_{\Omega} \sqrt{\rho_k} dx \leq \tau\left\|u_k \right\|_{2,\Omega} \left\|\sqrt{\rho_k}\right\|_{2,\Omega} + \left\|\nabla u_k \right\|^{(p-1)}_{p,\Omega} \left\|\nabla \sqrt{\rho_k} \right\|_{2,\Omega} + c
\end{equation}
Then, by the Sobolev Inequality,
\begin{equation}
	\begin{aligned}
	\int_{\Omega} \left|\sqrt{\rho_k}\right|^2 dx & \leq c\left( \int_{\Omega} \left|\sqrt{\rho_k}\right|^{\frac{2N}{N-2}} dx \right)^{\frac{N-2}{N}} \\
	& \leq c \int_{\Omega} \left|\nabla\sqrt{\rho_k}\right|^2 dx + c\left(\int_{\Omega} \sqrt{\rho} dx \right)^2 \\
	& \leq c \int_{\Omega} \left|\nabla\sqrt{\rho_k}\right|^2 dx + c\tau\left(\tau\max_{0 \leq t \leq T} \int_{\Omega} \buj^2 dx \right) \int_{\Omega} \left|\sqrt{\rho_k} \right|^2 dx \\
	&+ c\left(\max_{0 \leq t \leq T} \left\|\nabla \buj \right\|_{p,\Omega}\right)^{2(p-1)} \int_{\Omega} \left|\sqrt{\rho_k} \right|^2 dx
	\end{aligned}
\end{equation}
Then for sufficiently small $\tau$, (or sufficiently large $j$) we find,
\begin{equation}
	\int_{\Omega} \left|\sqrt{\rho_k}\right|^2 dx \leq  c \int_{\Omega} \left|\nabla\sqrt{\rho_k}\right|^2 dx + c\left(\max_{0 \leq t \leq T} \left\|\nabla \buj \right\|_{p,\Omega}\right)^{2(p-1)} \int_{\Omega} \left|\sqrt{\rho_k} \right|^2 dx.
\end{equation}
Summing up this equation over $k$ we find,
\begin{equation}
	\int_{\ot} \left|\sqrt{\brj}\right|^2 dxdt \leq c\left( 1 +  c\left(\max_{0 \leq t \leq T} \left\|\nabla \buj \right\|_{p,\Omega}\right)^{2(p-1)}\right) \int_{\ot} \left|\nabla\sqrt{\brj}\right|^2 dxdt. 
\end{equation}
The lemma then follows. 
\end{proof}

From lemmas \ref{lemma4.1} and \ref{lemma4.2} we can then conclude that the sequence $\{\sqrt{\brj}\}$ is bounded in $L^2\left(0,T;\wot\right)$. Then, from lemma \ref{lemma4.1} we also have that $\left|\nabla \buj \right|$ is bounded in $L^{\infty}\left(0,T; L^p\left(\Omega\right)\right)$.
\begin{lemma}\label{lemma4.3}
	The sequence $\{\tuj\}$, is precompact in the space $L^p\left(\ot\right)$. 
\end{lemma}
\begin{proof}
	We begin with the sequence $\{\tuj\}$. For any $t\in\left(0,T\right]$, there exists a $k$ so that, $t\in\left(t_{k-1} - t_k\right]$. Then using lemma \ref{lemma4.1} we have that, 
	\begin{align}
		\int_{\Omega} \left|\nabla \tuj(x,t)\right|^p dx & = \int_{\Omega} \left| \frac{t-t_{k-1}}{\tau} \nabla u_k + \left(1-\frac{t-t_{k-1}}{\tau}\right) \nabla u_{k-1} \right|^p dx \nonumber \\
		& \leq \frac{t-t_{k-1}}{\tau} \int_{\Omega} \left|\nabla u_k \right|^p dx + \left(1-\frac{t-t_{k-1}}{\tau}\right) \int_{\Omega} \left|\nabla u_{k-1} \right|^p dx \nonumber \\
		& \leq c \sup_{0\leq t \leq T} \int_{\Omega} \left|\nabla \buj \right|^p dx \leq c. 
	\end{align}
	Hence, we can conclude that there is a positive constant $c$, independent of $\tau$, so that,
	\begin{equation}\label{suplemma4.3}
		\sup_{0\leq t \leq T} \int_{\Omega} \left|\nabla \tuj(x,t)\right|^p dx \leq c. 
	\end{equation}
	Next, we multiply through \eqref{mthmfive} by $\tau$ and add the resulting equation to \eqref{mthmfour} to obtain,
	\begin{equation}
		\partial_t \tuj -\Delta \brj -\tau \Delta_p \buj + \tau^2 \buj = 0. 
	\end{equation}
	We then integrate this equation over $\Omega$ to find,
	\begin{equation}
		 \frac{d}{dt} \int_{\Omega} \tuj dx + \tau^2 \int_{\Omega} \buj dx = 0. 
	\end{equation}
	Integrating this equation with respect to $t$ and using lemma \ref{lemma4.1} we then obtain,
	\begin{equation}\label{l1fortuj}
		\left|\int_{\Omega} \tuj dx \right| \leq c. 
	\end{equation}
	We then use Poincare's inequality and \eqref{l1fortuj} to find,
	\begin{align}
		\int_{\Omega} \left|\tuj\right|^p dx & \leq c \int_{\Omega} \left| \tuj - \frac{1}{\left|\Omega\right|} \int_{\Omega} \tuj dx \right|^p dx + c \left(\int_{\Omega} \tuj dx \right)^p \\
		& \leq c \int_{\Omega} \left|\nabla \tuj \right|^p dx + c. 
	\end{align}
	Thus, from \eqref{suplemma4.3} we have that the sequence $\{\tuj\}$ is a bounded sequence in $L^{\infty}\left(0,T;W^{1,p}\left(\Omega\right)\right)$. 
	
	Next, since the sequence $\sqrt{\brj}$ is bounded in $L^2\left(0,T;\wot\right)$, we can conclude that, $\left|\nabla \rho\right| \in L^1\left(\ot\right)$. Also, $\tau \ln(\rho) \in L^2\left(\ot\right) $. Then, since,
	\begin{equation}
		\partial_t \tuj = \Delta \brj - \tau\ln(\brj)
	\end{equation}
	We can conclude that $\partial_t \tuj \in L^1\left(0,T;\left(W^{1,\infty}\left(\Omega\right)\right)^{*}\right)$. We are now in a position to be able to apply the Lions-Aubin lemma \ref{la}. Subsequently, we can conclude that $\tuj$ is precompact in $L^{p}\left(0,T; L^p\left(\Omega\right)\right)=L^p(\ot)$.
\end{proof}

To show that $\{\bar{u}_j\}$ is also precompact in $L^p\left(\ot\right)$ it is enough to show that it has the same limit pointwise a.e. on $\ot$. This can be accomplished first by noticing,
\begin{align*}
	\int_{0}^{T} \left(\tuj-\buj\right) dt &= \sum_{k=1}^{j} \left[ \int_{t_{k-1}}^{t_k} \frac{t-t_{k-1}}{\tau} \left(u_k-u_{k-1}\right) - \left(u_k-u_{k-1}\right) \right] dt \\
	&= -\frac{1}{2} \tau \sum_{k=1}^{j} \left(u_k-u_{k-1}\right) \\
	& = -\frac{1}{2} \tau \left( u_j - u_0\right).
\end{align*}
Now upon integrating the above equation over $\Omega$ and using lemma \ref{lemma4.3} we find,
\begin{align*}
	\int_{\ot} \left(\tuj-\buj\right) dx &\leq c \tau \sup_{0\leq t \leq T} \int_{\Omega} \tuj dx \\
	& \leq c\tau. 
\end{align*}
Subsequently, $\tuj$ and $\buj$ have the same pointwise limit. In light of this, along with lemmas \ref{lemma4.1}, and \ref{lemma4.2}, we can conclude that at least for a subsequence,
\begin{align}
	\tuj \rightarrow u \ \ &\mbox{ Weakly in $L^p\left(0,T;W^{1,p}\left(\Omega\right)\right)$, strongly in $L^p\left(\ot\right)$ and a.e. on $\ot$, } \label{tujpointwise} \\
	\buj\rightarrow u \ \ & \mbox{ Weakly in $L^p\left(0,T;W^{1,p}\left(\Omega\right)\right)$, strongly in $L^p\left(\ot\right)$ and a.e. on $\ot$, } \\
	\sqrt{\brj} \rightarrow \rho \ \ & \mbox{ Weakly in $L^2\left(0,T;\wot\right)$. }
\end{align}

Now, for the $\ln(\brj)$ term, we are able to gain the following $L^1\left(\ot\right)$ bound.

\begin{lemma}\label{lemma4.4}
	There exists a constant $c$ which is independent of $\tau$ so that,
	\begin{equation}
		\int_{\ot} \left|\ln(\brj)\right| dxdt \leq c. 
	\end{equation}
\end{lemma}
\begin{proof}
	For the proof we first integrate equation \ref{mthmfive} over $\Omega$ to find that,
	\begin{equation}
		\int_{\Omega} \ln(\brj) dx = \tau \int_{\Omega} \buj dx. 
	\end{equation}
	Then by lemma \ref{lemma4.1} we obtain,
	\begin{equation}
		\max_{0 \leq t \leq T} \left|\int_{\Omega} \ln(\brj) dx \right| \leq \tau \max_{0 \leq t \leq T} \int_{\Omega} \left|\buj\right| dx \leq c. 
	\end{equation}
	Then, using this we have,
	\begin{equation}
		\begin{aligned}
			\int_{\Omega} \left|\ln(\brj)\right| dx &= \int_{\Omega} \ln^+(\brj) dx + \int_{\Omega} \ln^-(\brj) dx \\
			&= 2\int_{\Omega} \ln^+(\brj) dx - \int_{\Omega} \ln(\brj) dx \\
			& \leq 2\int_{\Omega} \brj dx + c.
		\end{aligned}
	\end{equation}
	Upon integrating this inequality with respect to $t$, and using lemma \ref{lemma4.2}, we obtain the claim. 
\end{proof}

We are now ready to prove further compactness for the sequence $\{\nabla\buj\}$. 
\begin{lemma}\label{lemma4.5}
	The sequence $\{\nabla\buj\}$ is precompact in $\left(L^q\left(\ot\right)\right)^N$ for $q<p$. 
\end{lemma}
\begin{proof}
	The idea for this proof comes from \cite{XX}. From \ref{tujpointwise}, we can assume that $\buj \rightarrow u$ pointwise a.e. on $\ot$ (at least for a subsequence). Then, by Egoroff's Theorem, for each $\delta > 0$ there is $E \subset \ot$ so that,
	\begin{equation}
		\left|\ot\backslash E \right| \leq \delta, \ \ \mbox{ and $\buj \rightarrow u$ uniformly on $E$. }
	\end{equation}
	The uniform convergence allows us to find a positive number $c$ so that,
	\begin{equation}\label{bound}
		\left|\buj\right| \leq c \ \ \mbox{ on $E$. } 
	\end{equation}
	Then, for any $\varepsilon > 0$, we have that,
	\begin{equation}
		\left|\bar{u}_{j_1} - \bar{u}_{j_2} \right| < \varepsilon \ \ \mbox{ on $E$, for sufficiently large $j_1, j_2$. }
	\end{equation}
	We now define the function 
	\begin{equation}
		\gamma_{\varepsilon}(s) = \left\{ \begin{array}{ll} \varepsilon, & s > \varepsilon \\ s, & \left|s\right| \leq \varepsilon \\ -\varepsilon, & s < -\varepsilon \end{array} \right.
	\end{equation}
	Now, from equation \ref{mthmfour}, we derive, 
	\begin{equation}
		-\Delta_p \bar{u}_{j_1} + \Delta_p \bar{u}_{j_2} + \tau_1\bar{u}_{j_1} - \tau_2\bar{u}_{j_2}  = \ln(\bar{\rho}_{j_1}) - \ln(\bar{\rho}_{j_2}). 
	\end{equation}
	We then use $\gamma_{\varepsilon}\left(\bar{u}_{j_1} - \bar{u}_{j_2} \right)$ as a test function in the above equation to find,
	\begin{equation}
		\begin{aligned}
			\int_{\Omega} & \left( \left|\nabla\bar{u}_{j_1} \right|^{p-2} \nabla \bar{u}_{j_1} - \left|\nabla\bar{u}_{j_2} \right|^{p-2} \nabla \bar{u}_{j_2} \right) \cdot \nabla \gamma_{\varepsilon}\left(\bar{u}_{j_1} - \bar{u}_{j_2} \right) dx \\
			& \leq \int_{\Omega} \left( \ln(\bar{\rho}_{j_1}) - \ln(\bar{\rho}_{j_2}) \right) \gamma_{\varepsilon}\left(\bar{u}_{j_1} - \bar{u}_{j_2} \right) dx + \int_{\Omega} \left(\tau_2 \bar{u}_{j_2} - \tau_1 \bar{u}_{j_1} \right) \gamma_{\varepsilon}\left(\bar{u}_{j_1} - \bar{u}_{j_2} \right) dx
		\end{aligned}
	\end{equation}
	Now, upon integrating with respect to $t$ and using lemma \ref{odenone} and \eqref{bound} we derive that,
	\begin{equation}
		\begin{aligned}
			\int_{E} \left|\nabla\left(\bar{u}_{j_1} - \bar{u}_{j_2} \right) \right|^2 dxdt &\leq c \int_{\ot} \left(\ln(\bar{\rho}_{j_1}) - \ln(\bar{\rho}_{j_2}) \right) \gamma_{\varepsilon}\left(\bar{u}_{j_1} - \bar{u}_{j_2} \right) dxdt \\
			&+ \int_{\ot}\left(\tau_2 \bar{u}_{j_2} - \tau_1 \bar{u}_{j_1} \right) \gamma_{\varepsilon}\left(\bar{u}_{j_1} - \bar{u}_{j_2} \right)dxdt  \\
			& \leq c \varepsilon. 
		\end{aligned}
	\end{equation} 
	Subsequently, we may conclude that the sequence $\{\nabla \buj\}$ is precompact in $\left(L^2\left(\ot\right) \right)^{N}$. Then, let $q<p$ be given. Then we may estimate,
	\begin{align}
		\int_{\ot} \left| \nabla \left( \bar{u}_{j_1} - \bar{u}_{j_2} \right) \right|^q dxdt & = \int_{E} \left| \nabla \left( \bar{u}_{j_1} - \bar{u}_{j_2} \right) \right|^q dxdt + \int_{\ot\backslash E} \left| \nabla \left( \bar{u}_{j_1} - \bar{u}_{j_2} \right) \right|^q dxdt \nonumber \\
		& \leq c \left| \ot \backslash E \right|^{1-\frac{q}{p}} + \int_{\ot\backslash E} \left| \nabla \left( \bar{u}_{j_1} - \bar{u}_{j_2} \right) \right|^q dxdt \nonumber \\
		& \leq c \delta^{1-\frac{q}{p}} + \int_{\ot\backslash E} \left| \nabla \left( \bar{u}_{j_1} - \bar{u}_{j_2} \right) \right|^q dxdt
	\end{align}
	Therefore, we can conclude that,
	\begin{equation}
		\limsup_{j\rightarrow\infty} \int_{\ot} \left| \nabla \left( \bar{u}_{j_1} - \bar{u}_{j_2} \right) \right|^q dxdt \leq  c \delta^{1-\frac{q}{p}}.
	\end{equation}
	Since $\delta$ was arbitrary, the claim then follows. 
\end{proof}

Now, by lemma \ref{lemma4.1} we can conclude that $\left|\nabla \buj \right|^{p-2} \nabla \buj$ is bounded in $L^{\frac{p}{p-1}}\left(\ot\right)$, and (at least for a subseqence), $$\left|\nabla \buj \right|^{p-2} \nabla \buj \rightarrow \left|\nabla u \right|^{p-2}\nabla u \ \ \mbox{ pointwise a.e. on $\ot$. } $$ It then follows from the Lebesgue Dominated Convergence Theorem that, 
\begin{equation}
	\left|\nabla \buj \right|^{p-2} \nabla \buj \rightarrow \left|\nabla u \right|^{p-2}\nabla u \ \ \mbox{ Strongly in $L^{\frac{p}{p-1}}\left(\ot\right)$. }
\end{equation}
Now we need to deal with the nonlinear term $\ln(\brj)$. To handle this term we will first need to demonstrate the pointwise convergence of the sequence $\{\brj\}$. This is accomplished in our next lemma.
\begin{lemma}\label{lemma4.6}
	The sequence $\{\sqrt{\brj}\}$ is precompact in $L^2\left(\ot\right)$. 
\end{lemma}
\begin{proof}
	For the proof we first derive from equation \eqref{mthmfive} that,
	\begin{equation}
		\ln(\bar{\rho}_{j_1}) - \ln(\bar{\rho}_{j_2}) = \tau_1 \bar{u}_{j_1} - \tau_2\bar{u}_{j_2}  - \Delta_p \bar{u}_{j_1} + \Delta_p \bar{u}_{j_2}. 
	\end{equation}
	We now use $\left(\sqrt{\bar{\rho}_{j_1}} - \sqrt{\bar{\rho}_{j_2} } \right)$ as a test function in the above equation to find,
	\begin{equation}\label{meh}
		\begin{aligned}
		\lefteqn{	\int_{\ot} \left(\ln(\bar{\rho}_{j_1}) - \ln(\bar{\rho}_{j_2})\right) \left(\sqrt{\bar{\rho}_{j_1}} - \sqrt{\bar{\rho}_{j_2} } \right) dxdt}\\
		 & =  \int_{\ot} \left(\tau_1 \bar{u}_{j_1} - \tau_2\bar{u}_{j_2} \right) \left(\sqrt{\bar{\rho}_{j_1}} - \sqrt{\bar{\rho}_{j_2} } \right) dxdt \\
			& + \int_{\ot} \left( \left|\nabla \bar{u}_{j_1} \right|^{p-2} \nabla \bar{u}_{j_1} - \left|\nabla \bar{u}_{j_2} \right|^{p-2}\nabla \bar{u}_{j_2} \right) \cdot \nabla \left(\sqrt{\bar{\rho}_{j_1}} - \sqrt{\bar{\rho}_{j_2}} \right) dxdt .
		\end{aligned}
	\end{equation}
	Then, the two integrals on the right converge to zero, as they are products of weakly and strongly convergent sequences. For the left hand side however, we use the third statement in lemma \ref{inequalities} to get,
	\begin{equation}
		\int_{\ot} \left(\ln(\bar{\rho}_{j_1}) - \ln(\bar{\rho}_{j_2})\right) \left(\sqrt{\bar{\rho}_{j_1}} - \sqrt{\bar{\rho}_{j_2} } \right) dxdt \geq 4 \int_{\ot} \left(\bar{\rho}_{j_1}^{\frac{1}{4}} - \bar{\rho}_{j_2}^{\frac{1}{4}} \right)^2 dxdt.
	\end{equation}
	This, together with \eqref{meh}, implies that the sequence $\{\brj^{\frac{1}{4}}\}$ is precompact in $L^2\left(\ot\right)$. Therefore, at least a subsequence of $\{\brj^{\frac{1}{4}}\}$ converges pointwise a.e. on $\ot$. Then $\sqrt{\brj} = \left(\brj^{\frac{1}{4}} \right)^2$ must converge pointwise a.e. on $\ot$. By the Lebesgue Dominated Convergence Theorem, we can then conclude that at least a subsequence of $\{\sqrt{\brj}\}$ converges strongly in $L^2\left(\ot\right)$. 
\end{proof}

We are now ready to prove Theorem \ref{mainthm}. 
\begin{proof}
	
By passing to subsequences, we may assume that,
\begin{align}
	\tuj \rightarrow u \ \ &\mbox{ strongly in $L^p\left(\ot\right)$, }  \\
	\buj \rightarrow u \ \ &\mbox{ strongly in $L^p\left(0,T;W^{1,p}\left(\Omega\right)\right)$, and a.e. on $\ot$, }   \\
	\sqrt{\brj} \rightarrow \sqrt{\rho} \ \ &\mbox{ strongly in $L^2\left(\Omega\right)$, weakly in $L^2\left(0,T;\wot\right)$, and a.e. on $\ot$. } \label{rootrhoconverges}
\end{align}

Now, from lemma \ref{lemma4.6} we can also conclude that,
\begin{equation}
	\nabla \brj \rightarrow \nabla \rho \ \ \mbox{ weakly in $L^1(\ot)$, }. 
\end{equation}
To see this, we note that,
\begin{equation}
	\nabla \brj = 2\sqrt{\brj} \nabla \sqrt{\brj}. 
\end{equation}
Then, $\{\sqrt{\brj}\}$ converges strongly in $L^2(\ot)$ while $\{\nabla \sqrt{\brj}\}$ converges weakly in $\left(L^2\left(\ot\right)\right)^N$. Therefore their product converges weakly in $L^1(\Omega)$. 

Next, by Fatou's Lemma, and lemma \ref{lemma4.4} we have that,
\begin{equation}
	\int_{\ot} \left|\ln(\rho) \right| dxdt \leq \liminf_{j\rightarrow \infty} \int_{\ot} \left|\ln(\brj)\right| dxdt \leq c. 
\end{equation}
Then, the set,
\begin{equation}
	A_0 = \{ (x,t) \in \ot: \rho(x,t) = 0\},
\end{equation}
must have Lebesgue measure $0$. This, along with \eqref{rootrhoconverges}, then implies that,
\begin{equation}
	\ln(\brj) \rightarrow \ln(\rho) \ \ \mbox{ a.e. on $\ot$. } 
\end{equation}
Then, clearly
\begin{equation}
	\tau \buj \rightarrow 0 \ \ \mbox{ strongly in $L^p(\ot)$ and a.e. on $\ot$. }
\end{equation}
Then, from equation \eqref{mthmfive} we can conclude that, 
\begin{equation}
	- \Delta_p \buj \rightarrow \ln(\rho) \ \ \mbox{ a.e. on $\ot$ }. 
\end{equation}
On account of lemmas \ref{lemma4.4} and \ref{lemma4.1} the term $\Delta_p \buj$ is bounded in $L^1\left(\ot\right)$ and in $L^p\left(0,T;\left(W^{1,p}\left(\Omega\right)\right)^*\right)$. Subsequently, we have,
\begin{equation}
	-\Delta_p \buj \rightarrow -\Delta_p u \equiv \mu \ \ \mbox{ weakly in both $\mathbb{M}(\ot)$ and $L^p\left(0,T;\left(W^{1,p}\left(\Omega\right)\right)^*\right)$. }
\end{equation}
The question is then, do we have,
\begin{equation}
	-\Delta_p u = \mu = \ln(\rho). 
\end{equation}
Our next lemma addresses this question. 
\begin{lemma}\label{lemma4.7}
	The restriction of $\mu$ to the set $\overline{\ot}\backslash A_0$ is the function $\ln(\rho)$. In other words the Lebesgue Decomposition of $\mu$ is,
	\begin{equation}
		\mu = \ln(\rho) + \nu_s, 
	\end{equation}
	where $\nu_s$ is a measure supported in $A_0$, and we have,
	\begin{equation}
		\rho = e^{\mu} \ \ \mbox{ on the set $\overline{\ot}\backslash A_0$. } 
	\end{equation}

\end{lemma}
\begin{proof}
	For each $\varepsilon > 0$ we let $\theta_\varepsilon$ be a smooth function on $\mathbb{R}$ so that $\theta_\varepsilon(s) = 1$ if $s \geq 2\varepsilon$, and $\theta_\varepsilon(s) = 0$ if $s \leq \varepsilon$, and,
	\begin{equation}
		0 \leq \theta_\varepsilon \leq 1 \ \ \mbox{ on $\mathbb{R}$. } 
	\end{equation}
	Then, we have, 
	\begin{equation}
		\theta_\varepsilon(\sqrt{\brj}) \rightarrow \theta_\varepsilon(\sqrt{\rho})\ \ \mbox{ strongly in $L^q\left(\ot\right)$, for each $q\geq1$.  } 
	\end{equation}
	Now, pick a function $\xi \in C^{\infty}\left(\overline{\ot}\right)$. Multiply through equation \eqref{mthmfive} by $\theta_\varepsilon(\sqrt{\brj}) \xi$ to obtain,
	\begin{equation}\label{grrr}
		- \int_{\ot} \Delta_p \buj \theta_\varepsilon(\sqrt{\brj})\xi dxdt + \tau \int_{\ot} \buj \theta_\varepsilon(\sqrt{\brj}) \xi dxdt = \int_{\ot}\ln(\brj)\theta_\varepsilon(\sqrt{\brj})\xi dxdt.
	\end{equation}
	Then, for each $\varepsilon$ the sequence $\{\theta_\varepsilon(\sqrt{\brj})\ln(\brj)\}$ is bounded in $L^2(\ot)$. This along with \eqref{rootrhoconverges} implies that,
	\begin{equation}
		\int_{\ot}\ln(\brj)\theta_\varepsilon(\brj)\xi dxdt \rightarrow \int_{\ot} \ln(\rho)\theta_\varepsilon(\rho) \xi dxdt. 
	\end{equation}
	On the other hand,
	\begin{equation}
		-\int_{\ot} \Delta_p \buj \theta_\varepsilon(\brj) \xi dxdt = \int_{0}^{T} \langle -\Delta_p \buj , \theta_\varepsilon(\sqrt{\brj})\xi \rangle dt \rightarrow \int_{\ot} \theta_\varepsilon(\sqrt{\rho}) \xi d\mu. 
	\end{equation}
	Then, we can pass to the limit as $j\rightarrow\infty$ in \eqref{grrr} to find,
	\begin{equation}
		\int_{\ot} \theta_\varepsilon(\sqrt{\rho}) \xi d\mu = \int_{\ot} \theta_\varepsilon(\sqrt{\rho}) \ln(\rho) \xi dxdt. 
	\end{equation}
	Now, using the Dominated Convergence Theorem, we let $\varepsilon \rightarrow 0$ to find,
	\begin{equation}
		\int_{\ot\backslash A_0} \xi d\mu = \int_{\ot\backslash A_0 } \ln(\rho) \xi dxdt.  
	\end{equation}
	Since this is true for all $\xi \in C^{\infty}\left(\ot\right)$, we can conclude that,
	\begin{equation}
		\mu = \ln(\rho) \ \ \mbox{ on $\ot \backslash A_0$. }
	\end{equation}
	The proof is complete. 
\end{proof}
Now that we have lemma \ref{lemma4.7} we are ready to pass to the limit in equations \eqref{mthmfour}-\eqref{mthmfive}. The proof of Theorem \ref{mainthm} is now complete. 
\end{proof}


\begin{thebibliography}{}
		
		
		
		
		
		\bibitem{BCF} W. K. Burton, N. Cabrera, and F. C. Frank, {\em The growth of crystals and the equilibrium structure of their surfaces}, Philosophical Trans. Royal soc. London A: Mathematical, Physical and Engineering Sciences, {\bf 243} (1951), no. 866, 299-358. 
		
		\bibitem{YG} Yuan Gao, {\em Global strong solution with BV derivatives to singular solid-on-solid model with exponential nonlinearity}, J. Differential Equations, {\bf 267} (2019), 4429-4447. 
		
		\bibitem{GLL} Y. Gao, J.-G. Liu, and X. Y. Lu, {\em Gradient flow approach to an exponential thin film equation: global existence and latent singularity}, ESAIM: Control, Optimisation and Calculus of Variations, {\bf 25} (2019), 49-. arXiv:1710.06995.
		
		\bibitem{GLL2} Y. Gao, J.-G. Liu, and J. Lu, {\em Weak solutions of a continuum model for vicinal surface in the ADL regime}, SIAM J. Math. Anal., {\bf 49} (2017), 1705-1731. 
		
		\bibitem{GT} D. Gilbarg and N.S. Trudinger, {\em Elliptic Partial Differential Equations of Second Order}, Springer-Verlag, Berlin, (1983).
		
		\bibitem{BM} R. Granero-Belinchón and M. Magliocca, {\em Global existence and decay to equilibrium for some crystal surface models}, {\bf 39} (2019), 2101-2131.
		
		\bibitem{KDM} J. Krug, H.T. Dobbs, and S. Majaniemi, {\em Adatom mobility for the solid-on-solid model}, Z. Phys. B {\bf 97} (1995), 281-291.
		
		\bibitem{LS} J.-G. Liu and R. Strain, {\em Global stability for solutions to the exponential PDE describing epitaxial growth}, Interfaces and Free Boundaries, {\bf 21} (2019), 61-68. 
		
		\bibitem{LX} J.-G. Liu, and X. Xu, {\em Existence Theorems For A Multidimensional Crystal Surface Model, } SIAM J. Math. Anal., {\bf 48} (2016), 3667-3687. 
		
		\bibitem{MK} D. Margetis and R. V. Kohn, {\em Continuum relaxation of interacting steps on crystal surfaces in $2+1$ dimensions}, Multiscale Modeling \& Simulation, {\bf 5} (2006), no. 3, 729-758.
		
		\bibitem{MW} J.L. Marzuola and J. Weare, {\em Relaxation of a family of broken-bond crystal surface models},Physical Review, E {\bf 88} (2013), 032403. 
		
		\bibitem{JOden} J. T. Oden, {\em Qualitative Methods in Nonlinear Mechanics}, Prentice-Hall, Inc, New Jersey, 1986. 
		
		\bibitem{PX} B. C. Price and X. Xu, {\em Strong solutions to a fourth order exponential PDE describing epitaxial growth}, Journal of Differential Equations, {\bf 306} (2022) 220-250.
		
		\bibitem{S} J. Simon, {\it Compact sets in the space $L^p(0,T;B)$}, Ann. Mat. Pura Appl., {\bf 146}(1987), 65-96.
		
		\bibitem{XX} X. Xu, {\em Partial Regularity for an Exponential PDE in Crystal Surface Models}, Nonlinearity, {\bf 35} 4392, (2022). 
		
		\bibitem{XX2} X. Xu, {\em
		Mathematical validation of a continuum model for relaxation of interacting steps in crystal surfaces in 2 space dimensions},  Calc. Var., {\bf  59}, 158, (2020).
	
		\bibitem{XX3} X. Xu, {\em Existence Theorems for a Crystal Surface Model Involving the $p$
			-Laplace Operator},   SIAM J. Math. Anal., {\bf 50} (2018), 4261–4281.
			

		
	

		

		

		
	
		
		

		

		

		

		

		

	\end{thebibliography}
\end{document}